\documentclass[a4paper, 11pt,reqno]{amsart}

\usepackage[T1]{fontenc}      

\usepackage[british]{babel}   
\usepackage{csquotes}         
\usepackage[babel]{microtype} 

\usepackage{mathtools}        
\usepackage{amssymb}          
\usepackage{amsthm}           
\usepackage{amstext}          
\usepackage{mleftright}       
\usepackage{gensymb}          
\usepackage[makeroom]{cancel} 
\usepackage{mathdots}         
\usepackage{mathrsfs}         
\usepackage{dsfont}           
\usepackage[shortlabels]{enumitem}
\usepackage{verbatim}

\usepackage{stickstootext}
\usepackage[stickstoo,varbb]{newtxmath} 
\makeatletter
\DeclareFontFamily{U}{ntxmia}{\skewchar \font =127}
 \DeclareFontShape{U}{ntxmia}{m}{it}{
                        <-> \ntxmath@scaled ntxmia
                      }{}    
                      \DeclareFontShape{U}{ntxmia}{b}{it}{
                        <-> \ntxmath@scaled ntxbmia
                      }{}
\makeatother

\usepackage{graphicx}         
\usepackage{psfrag}           
\usepackage{caption}          
\usepackage{tikz}             
\usetikzlibrary{backgrounds}

\usepackage{booktabs}         

\usepackage{enumitem}         

\usepackage[margin=1in]{geometry} 

\usepackage[textsize=footnotesize]{todonotes}

\usepackage[numbers,sort]{natbib}

\makeatletter
\def\NAT@spacechar{~}
\makeatother
\newcommand{\urlprefix}{}

\usepackage{hyperref}              
\usepackage[capitalise]{cleveref}  
\hypersetup{colorlinks=true,
   citecolor=blue,
   filecolor=blue,
   linkcolor=blue,
   urlcolor=blue
}
\usepackage{url}

\makeatletter
\if@cref@capitalise
\crefname{figure}{figure}{figures}
\crefname{claim}{Claim}{Claims}
\else
\crefname{figure}{Figure}{Figures}
\crefname{claim}{claim}{claims}
\fi
\makeatother
\Crefname{figure}{Figure}{Figures}
\Crefname{claim}{Claim}{Claims}

\usepackage{algorithm}
\usepackage{algpseudocode}

\allowdisplaybreaks

\theoremstyle{plain}
\newtheorem{definition}{Definition}

\newtheorem{claim}[definition]{Claim}

\newtheorem{theorem}[definition]{Theorem}
\newtheorem{corollary}[definition]{Corollary}
\newtheorem{lemma}[definition]{Lemma}
\newtheorem{fact}[definition]{Fact}
\newtheorem{conjecture}[definition]{Conjecture}

\numberwithin{equation}{section}

\renewcommand{\binom}[2]{\ensuremath{\mleft(\kern-.1em\genfrac{}{}{0pt}{}{#1}{#2}\kern-.1em\mright)}}    
\newcommand{\inbinom}[2]{\ensuremath{\bigl(\kern-.1em\genfrac{}{}{0pt}{}{#1}{#2}\kern-.1em\bigr)}} 

\DeclareMathOperator\supp{supp}

\newcommand{\cF}{\mathcal{F}}

\newcommand{\PP}{\mathbb{P}}

\newcommand{\EE}{\mathbb{E}}

\newcommand{\Var}{\operatorname{Var}}
\newcommand{\Cov}{\operatorname{Cov}}

\newcommand{\lgr}{\left<g\right>}

\newcommand{\lGr}{\left<G\right>}

\newcommand{\NN}{\mathbb{N}}

\makeatletter
\def\moverlay{\mathpalette\mov@rlay}
\def\mov@rlay#1#2{\leavevmode\vtop{%
  \baselineskip\z@skip \lineskiplimit-\maxdimen
  \ialign{\hfil$\m@th#1##$\hfil\cr#2\crcr}}}
\newcommand{\charfusion}[3][\mathord]{
    #1{\ifx#1\mathop\vphantom{#2}\fi
        \mathpalette\mov@rlay{#2\cr#3}
      }
    \ifx#1\mathop\expandafter\displaylimits\fi}
\makeatother

\newcommand{\Hnkm}{H^{(k)}(n,m)}
\newcommand{\Hnkq}{H^{(k)}(n,q)}
\newcommand{\rhalf}{\lceil \frac{r}{2}\rceil}

\newcommand{\COMMENT}[1]{}
\renewcommand{\COMMENT}[1]{\footnote{\textcolor{blue!70!black}{#1}}} 
\newcommand{\COMNEW}[1]{}
\renewcommand{\COMNEW}[1]{\footnote{\textcolor{red!70!black}{#1}}} 

\title{fractional vs. expectation thresholds: random support case}

\author[T.~Fischer]{Thomas Fischer}
\author[Y.~Person]{Yury Person}

\address{Institut f\"ur Mathematik, Technische Universit\"at Ilmenau, 98684 Ilmenau, Germany} 

\email{thomas.fischer \,|\, yury.person@tu-ilmenau.de}

\date{\today}

\begin{document}

\begin{abstract}
A conjecture of Talagrand (2010) states that the so-called expectation and fractional expectation thresholds are always within at most some constant factor from each other. We prove for the unweighted case that this is a.a.s.\ true when the support is a random hypergraph.
\end{abstract}

\maketitle

\section{Introduction}

The study of threshold functions in random discrete structures is one of the central topics in probabilistic combinatorics. Much progress happened in the recent years when the so-called expectation threshold  conjectures of Kahn and Kalai~\cite{KK07} and of Talagrand~\cite{TalagrandOriginal} were solved in~\cite{PP24}, respectively in~\cite{FKNP19}. More precisely,  the work of Frankston, Kahn, Narayanan and Park~\cite{FKNP19} and the subsequent work by Park and Pham~\cite{PP24} established  the location of the threshold (up to at most a logarithmic factor) in relation to the quantities fractional expectation resp. and  expectation threshold. The quantitative relation between these expectation thresholds remains in general  mysterious and Talagrand  conjectured in~\cite{TalagrandOriginal} that these  thresholds are always at most some absolute constant factor apart. The purpose of this note is to investigate a `randomized' version of Talagrand's conjecture. 
 
 First we introduce the necessary definitions and state already known partial results. For a finite nonempty $n$-element set $X$, some natural $k\le n$ and $p\in[0,1]$ we define for $G\subseteq 2^X:=\{S\colon S\subseteq X\}$ its weight $w(G,p)$ and the upset $\lGr$ of $G$ as follows
\begin{alignat*}{100}
    w(G,p):=\sum_{S\in G} p^{|S|}, \quad \lGr:=\{S\subseteq X\colon \exists\, T\in G, T\subseteq S\}.
\end{alignat*}
 The set $G$ and its upset $\lGr$ may in general represent some monotone property which contains some other monotone property in an easy way. There is also its fractional version, where one looks at a function $g\colon 2^X\to[0,1]$ instead and defines its weight $w(g,p)$ and `upset' $\lgr$ through
\begin{alignat}{100}\label{eq:wgp}
    w(g,p):=\sum_{S\in 2^X} g(S)p^{|S|}, \quad \lgr:=\{S\subseteq X\colon \sum_{T\subseteq S} g(T)\ge 1\}.
\end{alignat}
The expectation threshold $q(\cF)$ of some nontrivial monotone property $\cF\subseteq 2^X$ is defined as the largest $p$ such that there exists $G\subseteq 2^X$ with $\cF\subseteq \lGr$ and $w(G,p)\le 1/2$ and the fractional expectation threshold  $q_f(\cF)$  of $\cF$ is the largest $p$ such that there exists $g\colon 2^X\to [0,1]$ with $\cF\subseteq \lgr$ and $w(g,p)\le 1/2$. Talagrand~\cite{TalagrandOriginal} conjectured that there exists some absolute constant $L>0$ such that  $q_f(\cF)\le L \cdot q(\cF)$ always holds (the inequality $q(\cF)\le q_f(\cF)$ holds trivially). One of the equivalent forms of Talagrand's conjecture is the following.  

\begin{conjecture}[Conjecture~6 from~\cite{FP23}]\label{conj:Talagrand}
    There exists some fixed $L>1$ such that for all $n\in\NN$, $g: 2^X \rightarrow [0,1]$ with $g(\emptyset)=0$ and $p\in [0,1]$ the following holds. 
    If $w(g,p) = 1$ then there exists a set $G\subseteq 2^X\setminus\{\emptyset\}$ with $\left<g\right>\subseteq \left<G\right>$ and $w\left(G,\frac{p}{L}\right) \leq 1$.   
\end{conjecture}

Recently, Dubroff, Kahn and Park~\cite{DKP24} and Pham~\cite{Pham24} proved that Conjecture~\ref{conj:Talagrand} is true for any function $g\colon 2^X\to[0,1]$ which is supported on sets of constant size (i.e.\ independent of $n=|X|$). Moreover, from the quantitative result of Pham~\cite{Pham24}, combined with our result~\cite[Theorem~10]{FP23}, it follows that $L$ can be chosen as $O(\log\log|X|)$. Apart from this result, there were several test cases shown before by Talagrand himself ($g$ is supported on 1-element sets), by DeMarco and Kahn ($g$ is constant and is supported on the edge sets of the cliques of some graph of order $k$), by Frankston, Kahn and Park in~\cite{TalagrandTwoSets} ($g$ is supported on $2$-element sets) and by the authors~\cite{FP23,FP25} ($g$ is supported on the edges of  `almost' linear hypergraphs of any uniformity; $g$ is supported on the edge sets of the cliques of some uniform hypergraph). For more background and somewhat technical special cases we refer to~\cite{FP23, TalagrandOriginal}. It was also shown in~\cite{FP23} that it is enough to study functions $g\colon \binom{X}{k}\to[0,1]$ for all $k\le \ln(|X|)$. 

In this note we will consider constant functions $g$ whose support is some \emph{random} $k$-uniform hypergraph itself. Thus, we will be interested in the `average case' form of Conjecture~\ref{conj:Talagrand}.

The first model of a random hypergraph is a so-called binomial random hypergraph $\Hnkq$, which is a hypergraph on an $n$-element  vertex set $X$, where each of the $\binom{n}{k}$ possible edges appears independently with probability $q$. Another model is the random hypergraph  $\Hnkm$ which is a $k$-uniform hypergraph on an $n$-element vertex set $X$ with $m$ edges chosen uniformly at random from all such hypergraphs on $X$. Both models were studied extensively in other contexts since they generalize the two widely studied random graph models $G(n,p)$ and $G(n,m)$ in a straightforward way. 

Our main result is the following. 

\begin{theorem}\label{theorem:main}
    Let $k$, $r\in\NN$ and set  $g:= \frac{1}{r}\cdot\mathds{1}_{\Hnkm}$. Then for  $L\geq 4\cdot e^3$ the following holds a.a.s.\footnote{We are considering probabilities of events which tend to $1$ as $n$ goes to infinity.} 
 
    If  $p\in[0,1]$ is such that  $w(g,p)=1$ then there exists $G\subseteq 2^X\setminus\{\emptyset\}$ such that $\left<g\right>\subseteq \left<G\right>$ and $w\left(G,\frac{p}{L}\right)\leq 1$.
\end{theorem}

A similar statement holds for the binomial random hypergraph $\Hnkq$, which is a simple consequence of Theorem~\ref{theorem:main}.

\begin{corollary}\label{TheoremHknq}
    Let $k$, $r\in\NN$ and set  $g:= \frac{1}{r}\cdot\mathds{1}_{\Hnkq}$. Then for  $L\geq 4\cdot e^3$ the following holds a.a.s.
    
    If  $p\in[0,1]$ is such that  $w(g,p)=1$ then there exists $G\subseteq 2^X\setminus\{\emptyset\}$ such that $\left<g\right>\subseteq \left<G\right>$ and $w\left(G,\frac{p}{L}\right)\leq 1$.
\end{corollary}

\begin{proof}  
    For $g:= \frac{1}{r}\cdot\mathds{1}_{\Hnkq}$, let $p\in[0,1]$ be such that $w(g,p)=1$ (in the case that $w(g,1)<1$ we set $p=1$). Observe that, in this case, $p$ is itself a random variable. Let further 
    $G(g,L)$ be a $G\subseteq 2^X$ that minimizes $w\left(G,\frac{p}{L}\right)$ under the condition that $\lgr\subseteq\lGr$. 

    Let $Y$ be the number of edges in $\Hnkq$. If we condition on $Y=m$ then we get that $\Hnkq$ has the same distribution as  $\Hnkm$. Hence, we get by Theorem~\ref{theorem:main} that (uniformly over all $m$, since $m$ is allowed to depend on $n$)
    \begin{alignat*}{100} 
        \PP\left(\left.w\left(G(g,L),\frac{p}{L}\right)\geq 1 \right| Y=m\right) \underset{n\rightarrow\infty}{\longrightarrow} 0.
    \end{alignat*}
    Therefore we get with the tower property
    \begin{alignat*}{100}
        \PP\left(w\left(G(g,L),\frac{p}{L}\right)\geq 1\right) = \EE\left[\PP\left(\left.w\left(G(g,L),\frac{p}{L}\right)\geq 1\right| Y\right) \right] \underset{n\rightarrow\infty}{\longrightarrow} 0.   
    \end{alignat*}
\end{proof}

\section{Helpful Estimates}
Below we collect several helpful estimates which we are going to use.

\begin{fact}\label{fact:binomial_bounds}
    For nonnegative integers $j\leq k\leq m\leq n$ we have 
    \begin{alignat*}{100}
        \binom{n}{k}\le \left(\frac{en}{k}\right)^k, \quad\quad 
        \frac{\binom{m}{k}}{\binom{n}{k}}  \leq \left(\frac{m}{n}\right)^k, \quad\quad
        \frac{\binom{n-j}{k-j}}{\binom{n}{k}} \leq \left(\frac{k}{n}\right)^j.
    \end{alignat*}
\end{fact}

We write  $a\land b := \min(a,b)$ and $a\lor b := \max(a,b)$ for the sake of smaller formulas. 

\begin{lemma}\label{LemmaHypergeoProb}
    For $K,m\leq N$ and a hypergeometrically $(N,K,m)$-distributed random variable $Y$ the following holds  for $0\le y\le m\land K$:
    \begin{alignat*}{100}
        \mathbb{P}\left(Y\geq y\right) = \sum_{j=y}^{m\land K} \frac{\binom{K}{j}\cdot \binom{N-K}{m-j}}{\binom{N}{m}} \leq \left(\frac{K}{N}\right)^y\cdot \binom{m}{y}.
    \end{alignat*}
\end{lemma}
\begin{proof}
    We start with the definition of the hypergeometric$(N,K,m)$-distribution and then rewrite and estimate it appropriately as follows
    \begin{alignat*}{100}
        \mathbb{P}\left(Y\geq y\right) & = \sum_{j=y}^{m\land K} \frac{\binom{K}{j}\cdot \binom{N-K}{m-j}}{\binom{N}{m}}\\
        & = \sum_{j=y}^{m\land K} \frac{\binom{K}{y}\cdot \binom{K-y}{j-y}}{\binom{j}{y}}\cdot \binom{N-K}{m-j}\cdot \frac{\binom{m}{y}}{\binom{N}{y}\cdot \binom{N-y}{m-y}}\\
        & = \frac{\binom{K}{y}\cdot \binom{m}{y}}{\binom{N}{y}}\cdot \sum_{j'=0}^{(m-y)\land (K-y)} \frac{1}{\binom{y+j'}{y}}\cdot \frac{\binom{K-y}{j'}\cdot\binom{N-K}{m-j'-y}}{\binom{N-y}{m-y}}\\
        & \leq \frac{\binom{K}{y}}{\binom{N}{y}}\cdot \binom{m}{y}\cdot \sum_{j'=0}^{(m-y)\land (K-y)} \frac{\binom{K-y}{j'}\cdot\binom{(N-y)-(K-y)}{(m-y)-j'}}{\binom{N-y}{m-y}}\\
        & \overset{\text{Fact~\ref{fact:binomial_bounds}}}{\leq} \left(\frac{K}{N}\right)^y\cdot \binom{m}{y}.
    \end{alignat*}
    Where in the last step we used, additionally to Fact~\ref{fact:binomial_bounds}, that the sum goes over the weights of the hypergeometric $(N-y,K-y,m-y)$-distribution and therefore equals $1$.
\end{proof}

We will also need the following intuitive but somewhat technical  result about non positive correlation of some particular two hypergeometrically distributed and stochastically dependent random variables. 

\begin{lemma}\label{LemmaCovarianceDisjointHyperGeo}
Let $m$, $K$, $N\in \NN$ with $m$, $K\le N$. Let $X$ be an $N$-element set and let $S_1$, $S_2$ be disjoint subsets of $X$ of size $K$ each. Let $M$ be a set of  $m$ elements from $X$ drawn randomly without replacement and let $Y_i:=|S_i\cap M|$ for $i=1$, $2$. 
For any $r\in \NN_0$ the following holds
    \begin{alignat*}{100}
        \Cov\left(\mathds{1}_{\left\{Y_1\geq r\right\}},\mathds{1}_{\left\{Y_2\geq r\right\}}\right) \leq 0.
    \end{alignat*}
\end{lemma}
\begin{proof}
    First we can assume that  $0< r\leq K\land m$ since otherwise the events have probability $0$ or $1$, which leads to the covariance being $0$. Observe also that each $Y_i$ is hypergeometrically $(N,K,m)$-distributed. 

    Next we consider drawing $Y_1$ and $Y_2$ as a two step process where we first draw the variable $Y_1$ and then $Y_2$. Let $0\le t\le K\land m$. Given $Y_1=t$, we have that $Y_2$ is hypergeometrically $(N-K,K,m-t)$-distributed. Therefore, the probability of $Y_2\geq r$ given $Y_1=t$ is monotonically decreasing in $t$ as there are fewer elements left to draw into $S_2$. 
    With this observation we compute:
    \begin{alignat*}{100}
        \Cov\left(\mathds{1}_{\left\{Y_1\geq r\right\}},\mathds{1}_{\left\{Y_2\geq r\right\}}\right) & = \mathbb{E}\left[\mathds{1}_{\left\{Y_1\geq r\right\}}\cdot\mathds{1}_{\left\{Y_2\geq r\right\}}\right] - \mathbb{E}\left[\mathds{1}_{\left\{Y_1\geq r\right\}}\right]\cdot \mathbb{E}\left[\mathds{1}_{\left\{Y_2\geq r\right\}}\right]\\
        & = \mathbb{P}\left(Y_1, Y_2\geq r\right) - \mathbb{P}\left(Y_1\geq r\right)\cdot \mathbb{P}\left(Y_2\geq r\right)\\
        & = \mathbb{P}\left(Y_2\geq r\land Y_1\geq r\right) - \mathbb{P}\left(Y_1\geq r\right)\cdot \left(\mathbb{P}\left(Y_2\geq r\land Y_1\geq r\right) + \mathbb{P}\left(Y_2\geq r\land Y_1 < r\right)\right)\\
        & = \mathbb{P}\left(Y_1\geq r\right)\cdot \mathbb{P}\left(Y_2\geq r\mid Y_1\geq r\right) - \mathbb{P}\left(Y_1\geq r\right)\cdot \\
        & \quad \left(\mathbb{P}\left(Y_1\geq r\right)\cdot \mathbb{P}\left(Y_2\geq r\mid Y_1\geq r\right) + \left(1-\mathbb{P}\left(Y_1\geq r\right)\right)\cdot\mathbb{P}\left(Y_2\geq r\mid Y_1 < r\right)\right)\\
        & = \mathbb{P}\left(Y_1\geq r\right)\cdot \left(1- \mathbb{P}\left(Y_1\geq r\right)\right) \cdot \left(\mathbb{P}\left(Y_2\geq r\mid Y_1\geq r\right) - \mathbb{P}\left(Y_2\geq r\mid Y_1 < r\right)\right)\\
        & \leq \mathbb{P}\left(Y_1\geq r\right)\cdot \left(1- \mathbb{P}\left(Y_1\geq r\right)\right) \cdot \left(\mathbb{P}\left(Y_2\geq r\mid Y_1 = r\right) - \mathbb{P}\left(Y_2\geq r\mid Y_1 = r-1\right)\right)\\
        & \leq 0,
    \end{alignat*}
    where in the last step we used that $0\leq \mathbb{P}\left(Y_1\geq r\right)\leq 1$ and $\mathbb{P}\left(Y_2\geq r\mid Y_1 = r\right) \leq \mathbb{P}\left(Y_2\geq r\mid Y_1 = r-1\right)$. 
\end{proof}

\section{Proof of Theorem~\ref{theorem:main}}
Let $L\geq 4\cdot e^3$. Let $g:= \frac{1}{r}\cdot\mathds{1}_{\Hnkm}$ and $p\in[0,1]$ be such that $w(g,p)=1$ (in the case that $w(g,1)<1$ we have $\lgr =\emptyset$ and we are done with $G=\emptyset$). It follows from the definition of $w(g,p)$~\eqref{eq:wgp}, that  (we may assume that $m\geq r> 1$ since for $r\leq 1$ we can just take $G=\supp(g)$)
\begin{alignat}{100}\label{eq:p}
    p=\left(\frac{r}{m}\right)^{1/k}.
\end{alignat}

We will cover $\lgr$ by choosing $G$ as $G := G_0 \cup \bigcup_{j=k+1}^{k\cdot r - \rhalf} G_j$, where $G_j$s are defined in the following way:
\begin{itemize}
    \item $G_0:= \left\{\bigcup_{T\in U} T \ \left| \ U\in\binom{\supp(g)}{\rhalf}, \forall T_1,T_2\in U: T_1\cap T_2=\emptyset \lor T_1=T_2 \right.\right\}$ (so $G_0$ contains the unions of  all the matchings with $\rhalf$ edges; this is inspired by an approach in \cite{TalagrandCliques}),
    \item $G_j  := \binom{X}{j}$ if  $k+1\le  j \leq k\cdot r - \rhalf$ and $\frac{n\cdot p}{j} \le e^2$,
    \item $G_j  := \binom{X}{j}\cap \lgr$ if $k+1\le j \leq k\cdot r - \rhalf$ and $\frac{n\cdot p}{j} > e^2$.
\end{itemize}
\ \\

First we observe that $G$ indeed covers $\lgr$.
\begin{claim}\label{claim:covering}
    \begin{alignat*}{100}
        \lgr \subseteq \lGr
    \end{alignat*}
\end{claim}
\begin{proof}
    Assume we have an $S\in \left<g\right>$ that is not in $\left<G_0\right>$ (otherwise we are done). Then we know that $S$ contains at least $r$ many $T\in\supp(g)$ of which there are less than $\rhalf$ pairwise disjoint. So we take any $r$ of these sets $T$ and then observe that there are at most $\rhalf-1$ many pairwise disjoint sets.  So we get that the union of these $r$ sets  is at most $(\rhalf-1)\cdot k + (r-\rhalf+1)\cdot (k-1) \le  r\cdot k - \rhalf$. Hence $S$  is already contained  in one of the $G_j$ and therefore  $\lgr \subseteq \lGr$.
\end{proof}

So we are  left with calculating the weight $w\left(G,\frac{p}{L}\right)$. This task will be split into several claims.

\begin{claim}[weight of $G_{0}$]\label{claim:weightG0}
    \begin{alignat*}{100}
        w\left(G_{0},\frac{p}{L}\right) \le \left(\frac{2\cdot e}{L}\right)^{k\cdot \rhalf}
    \end{alignat*}
\end{claim}
\begin{proof}
    \begin{alignat*}{100}
        w\left(G_{0},\frac{p}{L}\right) & \le \binom{m}{\rhalf}\cdot \left(\frac{p}{L}\right)^{k\cdot \rhalf} 
        \overset{\eqref{eq:p}}{\le} \left(\frac{e\cdot m}{\rhalf}\right)^{\rhalf}\cdot \frac{r^{\rhalf}}{m^{\rhalf}\cdot L^{k\cdot\rhalf}} \leq \left(\frac{2\cdot e}{L}\right)^{k\cdot \rhalf}.
    \end{alignat*}
\end{proof}

To compute the weight of the $G_{j}$'s we will first separate the cases $\frac{n\cdot p}{j} \leq e^2$ and $\frac{n\cdot p}{j} > e^2$.

\begin{claim}[weight of $G_{j}$ for $\frac{n\cdot p}{j} \leq e^2$]\label{claim:weightGj_small_npj}
    For $k+1\le  j \leq k\cdot r - \rhalf$ and $\frac{n\cdot p}{j} \leq e^2$ we have 
    \begin{alignat*}{100}
        w\left(G_{j},\frac{p}{L}\right)\leq \left(\frac{e^3}{L}\right)^{j}.
    \end{alignat*}   
\end{claim}
\begin{proof}
    For $\frac{n\cdot p}{j} \leq e^2$ we have $G_j = \binom{X}{j}$ and therefore
    \begin{alignat*}{100}
        w\left(G_{j},\frac{p}{L}\right) = \binom{n}{j}\cdot \left(\frac{p}{L}\right)^j \leq \left(\frac{e\cdot n}{j}\right)^j\cdot \left(\frac{p}{L}\right)^j = \left(\frac{e}{L}\cdot \frac{n\cdot p}{j}\right)^j \leq \left(\frac{e^3}{L}\right)^{j}.
    \end{alignat*}
\end{proof}

In the case $\frac{n\cdot p}{j} > e^2$ we will need some first and second moment estimates.

\begin{claim}[expected weight of $G_{j}$ for $\frac{n\cdot p}{j} > e^2$]\label{claim:expectedweight_Gj_large_npj}
    For $k+1 \le j \leq k\cdot r - \rhalf$ and $\frac{n\cdot p}{j} > e^2$ we have 
    \begin{alignat*}{100}
        \EE\left[w\left(G_{j},\frac{p}{L}\right)\right]\le \left(\frac{e^2}{L}\right)^{j}\cdot \left(\frac{e\cdot j}{n\cdot p}\right)^{\rhalf}\le \left(\frac{e^2}{L}\right)^{j}.
    \end{alignat*}   
\end{claim}
\begin{proof}
    For $\frac{n\cdot p}{j} > e^2$ we can write $w\left(G_{j},\frac{p}{L}\right)=\sum_{S\in\binom{X}{j}}\mathds{1}_{\left\{\left|\Hnkm\cap\binom{S}{k}\right|\geq r\right\}}\cdot \left(\frac{p}{L}\right)^j$ and since we know that $\left|\Hnkm\cap\binom{S}{k}\right|$ is hypergeometricaly $\left(\binom{n}{k},\binom{j}{k},m\right)$-distributed, we get:
    \begin{alignat*}{100}
        \EE\left[w\left(G_{j},\frac{p}{L}\right)\right] & = \sum_{S\in\binom{X}{j}} \PP\left(\left|\Hnkm\cap\binom{S}{k}\right|\geq r\right)\cdot \left(\frac{p}{L}\right)^j\\
        & = \binom{n}{j}\cdot \left( \sum_{\ell=r}^{m\land\binom{j}{k}}\frac{\binom{\binom{j}{k}}{\ell}\cdot\binom{\binom{n}{k}-\binom{j}{k}}{m-\ell}}{\binom{\binom{n}{k}}{m}} \right) \cdot \frac{p^j}{L^j}\\
        & \overset{\text{Lemma~\ref{LemmaHypergeoProb}}}{\leq} \binom{n}{j}\cdot \left(\frac{\binom{j}{k}}{\binom{n}{k}}\right)^r \cdot\binom{m}{r} \cdot \frac{p^j}{L^j}\\
        & \overset{\text{Fact}~\ref{fact:binomial_bounds}}{\le} \left(\frac{e\cdot n}{j}\right)^j\cdot \left(\frac{j}{n}\right)^{k\cdot r} \cdot \left(\frac{e\cdot m}{r}\right)^r\cdot \frac{p^j}{L^j}\\
        & \overset{\eqref{eq:p}}{=}  \left(\frac{e\cdot n}{j}\right)^j\cdot \left(\frac{j}{n}\right)^{k\cdot r} \cdot \left(\frac{e}{p^k}\right)^r\cdot \frac{p^j}{L^j} \\
        & \le \left(\frac{e}{L}\cdot \frac{n\cdot p}{j}\right)^j\cdot \left(\frac{e\cdot j}{p\cdot n}\right)^{k\cdot r}\\
        & = \left(\frac{e^2}{L}\right)^j\cdot \left(\frac{e\cdot j}{n\cdot p}\right)^{k\cdot r-j} \leq \left(\frac{e^2}{L}\right)^{j}\cdot \left(\frac{e\cdot j}{n\cdot p}\right)^{\rhalf}.
    \end{alignat*}
    Where in the last step we used that $\frac{n\cdot p}{j} > e^2$ and $j \leq k\cdot r - \rhalf$.
\end{proof}

One might already see that the expectation $\EE\left[w\left(G_{j},\frac{p}{L}\right)\right]$ goes to $0$ if $j$ or $\frac{n\cdot p}{j}$ goes to infinity (or even just $r$). 
Therefore we will split further considerations into  the following three subcases: $ j\ge \ln(\ln(n))$, $\frac{n\cdot p}{j}\ge \ln(n)$ and when the former two do not hold (i.e.\ $j< \ln(\ln(n))$ and  $\frac{n\cdot p}{j}< \ln(n)$).

\begin{claim}[weight of $G_j$ for $j\geq\ln(\ln(n))$]\label{claim:weight_Gj_large_j}
    For $k+1\le j \leq k\cdot r - \rhalf$ and $j\geq\ln(\ln(n))$ we have 
    \begin{alignat*}{100}
        \PP\left(w\left(G_{j},\frac{p}{L}\right) \geq \left(\frac{2\cdot e^3}{L}\right)^{j}\right)\leq \left(\frac{1}{2}\right)^{j}\cdot \frac{1}{\ln(n)}.
    \end{alignat*}
\end{claim}
\begin{proof}
    We use Markov's inequality and get:
    \begin{alignat*}{100}
        \PP\left(w\left(G_{j},\frac{p}{L}\right) \geq \left(\frac{2\cdot e^3}{L}\right)^{j}\right) \leq \frac{\EE\left[w\left(G_{j},\frac{p}{L}\right)\right]}{\left(\frac{2\cdot e^3}{L}\right)^{j}} \overset{\text{Claim~\ref{claim:expectedweight_Gj_large_npj}}}{\le} \frac{\left(\frac{e^2}{L}\right)^{j}}{\left(\frac{2\cdot e^3}{L}\right)^{j}} \leq \left(\frac{1}{2}\right)^{j}\cdot e^{-j} \leq \left(\frac{1}{2}\right)^{j}\cdot \frac{1}{\ln(n)}.
    \end{alignat*}
\end{proof}

\begin{claim}[weight of $G_{j}$ for $\frac{n\cdot p}{j}\ge \ln(n)$]\label{claim:weight_Gj_verylarge_npj}
    For $k+1\le  j \leq k\cdot r - \rhalf$ and $\frac{n\cdot p}{j}\geq \ln(n)$ we have 
    \begin{alignat*}{100}
        \mathbb{P}\left(w\left(G_{j},\frac{p}{L}\right) \geq \left(\frac{2\cdot e^3}{L}\right)^{j}\right)\le \left(\frac{1}{2}\right)^{j}\cdot \frac{1}{\ln(n)}.
    \end{alignat*}
\end{claim}
\begin{proof}
    We use Markov's inequality and get:
    \begin{alignat*}{100}
        \mathbb{P}\left(w\left(G_{j},\frac{p}{L}\right) \geq \left(\frac{2\cdot e^3}{L}\right)^{j}\right) \leq \frac{\mathbb{E}\left[w\left(G_{j},\frac{p}{L}\right)\right]}{\left(\frac{2\cdot e^3}{L}\right)^{j}} \overset{\text{Claim~\ref{claim:expectedweight_Gj_large_npj}}}{\leq} \frac{\left(\frac{e^2}{L}\right)^{j}\cdot \left(\frac{e\cdot j}{n\cdot p}\right)^{\rhalf}}{\left(\frac{2\cdot e^3}{L}\right)^{j}} \leq \left(\frac{1}{2}\right)^{j}\cdot \frac{1}{\ln(n)}.
    \end{alignat*}
\end{proof}

Finally we may assume that $j< \ln(\ln(n))$ and  $\frac{n\cdot p}{j}< \ln(n)$ and we aim to estimate the variance of $w\left(G_{j},\frac{p}{L}\right)$ (Claim~\ref{claim:variance_weight_Gj}) and with that the probability of large weight (Claim~\ref{claim:weight_Gj}).
\begin{claim}\label{claim:variance_weight_Gj}
    For $k+1\le j \leq k\cdot r - \frac{r}{2}$, $j\leq \ln(\ln(n))$ and $e^2< \frac{n\cdot p}{j}< \ln(n)$ we have for $n$ large enough
    \begin{alignat*}{100}
        \Var\left(w\left(G_{j},\frac{p}{L}\right)\right) \leq \left(\frac{2\cdot e}{L}\right)^j\cdot \frac{\EE\left[w\left(G_{j},\frac{p}{L}\right)\right]}{\ln(n)}.
    \end{alignat*}
\end{claim}
\begin{proof}
    Since $w\left(G_{j},\frac{p}{L}\right)=\sum_{S\in\binom{X}{j}}\mathds{1}_{\left\{\left|\Hnkm\cap\binom{S}{k}\right|\geq r\right\}}\cdot \left(\frac{p}{L}\right)^j$, we will need to estimate the covariances of the form $\Cov\left(B_1,B_2\right)$ for $B_i:=\mathds{1}_{\left\{\left|\Hnkm\cap\binom{S_i}{k}\right|\geq r\right\}}$ where $S_1$, $S_2\in\binom{X}{j}$. Since for $|S_1\cap S_2|<k$, the sets $S_1$ and $S_2$ have no potential edges from $\Hnkm$ in common, the covariance $\Cov\left(B_1,B_2\right)$ will not be positive (as proven in Lemma \ref{LemmaCovarianceDisjointHyperGeo}). For the case $|S_1\cap S_2|\ge k$ we will estimate the covariance of these Bernoulli-distributed random variables $B_1$ and $B_2$ by the expectation of one of these: $\Cov(B_1,B_2)=\EE[B_1B_2]-\EE[B_1]\EE[B_2]\le \EE[B_1]$.

    This leads to the following chain of estimates:
    \begin{alignat*}{100}
        \Var\left(w\left(G_{j},\frac{p}{L}\right)\right) & = \sum_{S_1\in\binom{X}{j}}\sum_{S_2\in\binom{X}{j}}\Cov\left(\mathds{1}_{\left\{\left|\Hnkm\cap\binom{S_1}{k}\right|\geq r\right\}}\cdot \left(\frac{p}{L}\right)^j,\mathds{1}_{\left\{\left|\Hnkm\cap\binom{S_2}{k}\right|\geq r\right\}}\cdot \left(\frac{p}{L}\right)^j\right)\\
        & = \sum_{S_1\in\binom{X}{j}}\sum_{S_2\in\binom{X}{j}}\Cov\left(\mathds{1}_{\left\{\left|\Hnkm\cap\binom{S_1}{k}\right|\geq r\right\}},\mathds{1}_{\left\{\left|\Hnkm\cap\binom{S_2}{k}\right|\geq r\right\}}\right)\cdot \left(\frac{p}{L}\right)^{2j}\\
        & \leq \binom{n}{j}\cdot\sum_{\ell=k}^{j}\binom{j}{\ell}\cdot\binom{n-j}{j-\ell}\cdot \mathbb{E}\left[\mathds{1}_{\left\{\left|\Hnkm\cap\binom{S}{k}\right|\geq r\right\}}\right]\cdot \left(\frac{p}{L}\right)^{2j}=:A,
    \end{alignat*}
    where $S$ is some arbitrary $j$-element subset of $X$. We can further rewrite the right hand side above and estimate as follows (for $n$ large enough):
    \begin{alignat*}{100}
        A & = \binom{n}{j}\cdot\sum_{\ell=k}^{j\land j}\frac{\binom{j}{\ell}\cdot\binom{n-j}{j-\ell}}{\binom{n}{j}}\cdot \mathbb{E}\left[w\left(G_{j},\frac{p}{L}\right)\right]\cdot \left(\frac{p}{L}\right)^{j}\\
        & \overset{\text{Lemma~\ref{LemmaHypergeoProb}}}{\leq} \binom{n}{j}\cdot \left(\frac{j}{n}\right)^k\cdot\binom{j}{k}\cdot \mathbb{E}\left[w\left(G_{j},\frac{p}{L}\right)\right]\cdot \left(\frac{p}{L}\right)^{j}\\
        & \leq \left(\frac{e\cdot n}{j}\right)^j\cdot \left(\frac{j}{n}\right)^k\cdot 2^j\cdot \mathbb{E}\left[w\left(G_{j},\frac{p}{L}\right)\right]\cdot \left(\frac{p}{L}\right)^{j}\\
        & = \frac{e^j\cdot 2^j}{L^j}\cdot \left(\frac{j}{n}\right)^k\cdot \left(\frac{n\cdot p}{j}\right)^j\cdot \mathbb{E}\left[w\left(G_{j},\frac{p}{L}\right)\right]\\
        & \leq \left(\frac{2\cdot e}{L}\right)^j\cdot \left(\frac{\ln(\ln(n))}{n}\right)^k\cdot \left(\ln(n)\right)^{\ln(\ln(n))}\cdot \mathbb{E}\left[w\left(G_{j},\frac{p}{L}\right)\right]\\
        & \le \left(\frac{2\cdot e}{L}\right)^j\cdot \frac{\mathbb{E}\left[w\left(G_{j},\frac{p}{L}\right)\right]}{\ln(n)}.
    \end{alignat*}
\end{proof}

\begin{claim}\label{claim:weight_Gj}
    For $k+1\le j \leq k\cdot r - \rhalf$, $j\leq \ln(\ln(n))$ and $e^2< \frac{n\cdot p}{j}< \ln(n)$ we have for $n$ large enough
    \begin{alignat*}{100}
        \PP\left(w\left(G_{j},\frac{p}{L}\right)\geq \left(\frac{2\cdot e^3}{L}\right)^{j}\right)\leq \left(\frac{1}{2}\right)^j\cdot\frac{1}{\ln(n)}.
    \end{alignat*}
\end{claim}
\begin{proof}
    Since, by Claim~\ref{claim:expectedweight_Gj_large_npj}, $\EE\left[w\left(G_{j},\frac{p}{L}\right)\right]\le \left(\frac{e^2}{L}\right)^{j}$ holds, we may use Chebyshev's inequality and $j\geq k+1\geq 2$ to get:
    \begin{alignat*}{100}
        \PP\left(w\left(G_{j},\frac{p}{L}\right)\geq \left(\frac{2\cdot e^3}{L}\right)^{j}\right) & \overset{\text{Claim~\ref{claim:expectedweight_Gj_large_npj}}}{\leq} 
        \PP\left(w\left(G_{j},\frac{p}{L}\right)\geq \EE\left[w\left(G_{j},\frac{p}{L}\right)\right] + 2^{j-1}\cdot\left(\frac{e^3}{L}\right)^{j}\right)\\
        & \leq \frac{\Var\left(w\left(G_{j},\frac{p}{L}\right)\right)}{2^{2j-2}\cdot\left(\frac{e^3}{L}\right)^{2j}} \overset{\text{Claim~\ref{claim:variance_weight_Gj}}}{\leq} \frac{\left(\frac{2\cdot e}{L}\right)^j\cdot \frac{\mathbb{E}\left[w\left(G_{j},\frac{p}{L}\right)\right]}{\ln(n)}}{2^j\cdot \left(\frac{e^3}{L}\right)^{2j}}\\
        & \leq  \left(\frac{1}{2}\right)^j\cdot \frac{\mathbb{E}\left[w\left(G_{j},\frac{p}{L}\right)\right]}{\left(\frac{e^3}{L}\right)^j}\cdot \frac{1}{\ln(n)} \overset{\text{Claim~\ref{claim:expectedweight_Gj_large_npj}}}{\leq} \left(\frac{1}{2}\right)^j\cdot\frac{1}{\ln(n)}.
    \end{alignat*}
\end{proof}

Now for the final step we combine Claims~\ref{claim:weightG0},~\ref{claim:weightGj_small_npj},~\ref{claim:weight_Gj_large_j},~\ref{claim:weight_Gj_verylarge_npj} and ~\ref{claim:weight_Gj} to bound the weight of $G$. Recall that 
\begin{alignat*}{100}
    G := G_0 \cup \bigcup_{j=k+1}^{k\cdot r - \rhalf} G_j.
\end{alignat*}
Furthermore, the sum below can be bounded by a geometric series as follows
\begin{alignat*}{100}
    \left(\frac{2\cdot e}{L}\right)^{k\cdot \rhalf}+\sum_{j=k+1}^{k\cdot r - \rhalf}\left(\frac{2\cdot e^3}{L}\right)^{j} \overset{L\geq 4\cdot e^3}{<}
    \left(\frac{2\cdot e}{L}\right)^{k}+ \frac{\left(\frac{2\cdot e^3}{L}\right)^{k+1}}{1-\frac{2\cdot e^3}{L}} \overset{L\geq 4\cdot e^3}{\le}
    \left(\frac{4\cdot e^3}{L}\right)^{k}.
\end{alignat*}
Thus,  we may  use the pigeonhole principle to get
\begin{alignat}{100}\label{eq:pigeonhole_estimate}
    \mathbb{P}\left(w\left(G,\frac{p}{L}\right)\ge \left(\frac{4\cdot e^3}{L}\right)^{k}\right) \le 
    \mathbb{P}\left(w\left(G_{0},\frac{p}{L}\right)\ge \frac{2\cdot e^3}{L}\right) + \sum_{j=k+1}^{k\cdot r - \rhalf} \mathbb{P}\left(w\left(G_{j},\frac{p}{L}\right)\ge \left(\frac{2\cdot e^3}{L}\right)^{j}\right).
\end{alignat}
By Claims~\ref{claim:weightG0},~\ref{claim:weightGj_small_npj},~\ref{claim:weight_Gj_large_j},~\ref{claim:weight_Gj_verylarge_npj} and ~\ref{claim:weight_Gj}, the probabilities are either zero or can be bounded each by $\left(\frac{1}{2}\right)^{j}\cdot \frac{1}{\ln(n)}$. The right hand side  of~\eqref{eq:pigeonhole_estimate} can thus be bounded by  $\frac{1}{\ln(n)}$:
\begin{alignat}{100}\label{eq:total_weight_G}
        \mathbb{P}\left(w\left(G,\frac{p}{L}\right)\ge \left(\frac{4\cdot e^3}{L}\right)^{k}\right) \le \frac{1}{\ln(n)}.
\end{alignat}

Since $\lgr\subseteq\lGr$ (Claim~\ref{claim:covering}) and, by~\eqref{eq:total_weight_G} and the choice of $L\geq 4\cdot e^3$, we a.a.s. \ have that the weight $w(G,\frac{p}{L})< \left(\frac{4\cdot e^3}{L}\right)^k\le 1$, this yields the claim of Theorem~\ref{theorem:main}. \qed

\section{Concluding Remarks}
Below we would like to share some thoughts on further possible generalizations of our Theorem~\ref{theorem:main}. 

We remark that Conjecture~\ref{conj:Talagrand} considers functions g which are weighted while our Theorem~\ref{theorem:main} deals only with the case when g is constant on its support. We believe that we could generalize our result as follows. 
\begin{enumerate}
    \item Let $W:=W(n)$ be any given random variable with values in $[0,1]$. Let $\left(W_T\right)_{T\in\binom{X}{k}}$ be i.i.d.\ copies of $W$. Then Theorem~\ref{theorem:main} (resp.\ Corollary~\ref{TheoremHknq}) still hold if we choose $g(T):= W_T$, $g(T):= W_T\cdot\mathds{1}_{\Hnkm}(T)$ resp.\ $g(T):= W_T\cdot\mathds{1}_{\Hnkq}(T)$. 
    \item  For any given $\hat{g}:\binom{X}{k}\rightarrow [0,1]$ and a random permutation $\pi:\binom{X}{k}\rightarrow\binom{X}{k}$ Theorem \ref{theorem:main} (resp. Corollary \ref{TheoremHknq}) still hold if we choose $g(T):= (\hat{g}\circ \pi)(T)$, $g(T):= (\hat{g}\circ \pi)(T)\cdot\mathds{1}_{\Hnkm}(T)$ resp.\ $g(T):= (\hat{g}\circ\pi)(T)\cdot\mathds{1}_{\Hnkq}(T)$. 
\end{enumerate}

The proofs would follow from~\cite[Theorem 23]{FP23} in conjunction with Claim~\ref{claim:expectedweight_Gj_large_npj}, which shows that  the expected weight is falling exponentially fast in $r$. A similar reduction from weighted to an unweighted case was first considered in~\cite{FKNP19}. We omit the details since these would require a lot of technical calculations without too much benefit, but we firmly believe that the unweighted case is the key to solve Conjecture~\ref{conj:Talagrand}. 

It would also be very interesting to turn the randomized setting of Theorem~\ref{theorem:main} into a \emph{pseudorandom} setting, although the appropriate definition of pseudorandomness is not clear. Still, this could provide an approach towards Conjecture~\ref{conj:Talagrand} by decomposing a given function $g$ as $g=g_1+g_2$ where $g_1$ a pseudorandom part  and $g_2$ is  an easily coverable part.

\end{document}